\documentclass[twoside]{article}
\usepackage[accepted]{aistats2014}
\usepackage{enumerate}
\usepackage{amsbsy}
\usepackage{amsmath,amsthm}
\usepackage{amsfonts}
\usepackage{latexsym}
\usepackage{graphicx}
\def\Alg{\mathsf{Alg}}

\newcommand{\defn}{\ensuremath{:  =}}

\def\bn{\big\|}

\def\R{\mathbb{R}}
\def\Pr{\mathbb{P}}
\def\<{\left\langle} % Angle brackets
\def\>{\right\rangle}

\newtheorem{theorem}{Theorem}[section]
\newtheorem{proposition}{Proposition}[section]

\newtheorem{remark}{Remark}[section]
\let\oldremark\remark
\renewcommand{\remark}{\oldremark\normalfont}

\usepackage{amssymb,pstricks,mathrsfs}
\usepackage{cite}
\usepackage{bbm}
\usepackage{enumerate}
\usepackage{paralist}
\usepackage{tikz}
\usetikzlibrary{shapes,decorations}
\usepgflibrary{shapes.geometric}
\usepackage{subfigure}
\usepackage{url}
\usepackage{multirow}
\usepackage{booktabs}
 \usepackage{epsfig}
 \usepackage{pst-grad} % For gradients
 \usepackage{arydshln}
\usepackage{algorithmic}
\usepackage[ruled,vlined,boxed]{algorithm2e}
\def\Path{\mathsf{PaTh}}
\def\Alg{\mathsf{Alg}}
\newcommand{\argmin}{\operatornamewithlimits{argmin}}

\def\bn{\big\|}
\begin{document}

\runningtitle{Path Thresholding: Asymptotically Tuning-Free
High-Dimensional Sparse Regression}


\twocolumn[

\aistatstitle{Path Thresholding: Asymptotically Tuning-Free \\
High-Dimensional Sparse Regression}

\aistatsauthor{ Divyanshu Vats and Richard G. Baraniuk }

\aistatsaddress{ Rice University }
]

\begin{abstract}

In this paper, we address the challenging problem of selecting tuning parameters for high-dimensional sparse regression.  We propose a simple and computationally efficient method, called path thresholding ($\Path$), that transforms \textit{any} tuning parameter-dependent sparse regression algorithm into an asymptotically tuning-free sparse regression algorithm.  More specifically, we prove that, as the problem size becomes large (in the number of variables and in the number of observations), $\Path$ performs accurate sparse regression, under appropriate conditions, without specifying a tuning parameter.  In finite-dimensional settings, we demonstrate that $\Path$ can alleviate the computational burden of model selection algorithms by significantly reducing the search space of tuning parameters.
\end{abstract}

\section{Introduction}

Sparse regression is a powerful tool used across several domains for estimating a sparse vector $\beta^*$ given linear observations $y = X \beta^* + w$, where $X$ is the known measurement matrix and $w$ is the observation noise.  Examples of applications include the analysis of gene expression data \cite{segal2003regression}, fMRI data \cite{varoquaux2012small}, and imaging data~\cite{CameraArray}.  Furthermore, sparse regression forms the basis of other important machine learning algorithms including dictionary learning \cite{olshausen1997sparse} and graphical model learning \cite{meinshausen2009lasso}.  

Several efficient algorithms now exist in the literature for solving sparse regression; see Lasso \cite{TibshiraniLasso1994}, OMP \cite{tropp2007signal}, and their various extensions \cite{ZouAdaptiveLasso2006,zhang2011adaptive,VatsBaraniukNIPS2013,VatsBaraniukSWAP2013}.  Several works have also analyzed the  conditions required for reliable identification of the sparse vector $\beta^*$; see \cite{buhlmann2011statistics} for a comprehensive review.  However, the performance of most sparse regression algorithms depend on a \textit{tuning parameter}, which in turn depends either on the statistics of the unknown noise in the observations or on the unknown sparsity level of the regression coefficients.  

Examples of methods to select tuning parameters include the Bayesian information criterion (BIC) \cite{schwarz1978estimating}, the Akaike information criterion (AIC) \cite{akaike1974new}, cross-validation (CV) \cite{geisser1975predictive,homrighausen2013lasso}, and methods based on minimizing the Stein's unbiased risk estimate (SURE)~\cite{donoho1995adapting,zou2007degrees,eldar2009generalized}.  All of the above methods are only suitable for low-dimensional settings, where the number of observations $n$ is much larger than the number of variables~$p$.  In high-dimensional settings, where $p > n$, all of the above methods typically overestimate the locations of the non-zero elements in $\beta^*$.  Although the overestimation could be empirically corrected using multi-stage algorithms \cite{ZouAdaptiveLasso2006,zhang2011adaptive,wasserman2009high,ZhangMulti2010}, this process is computationally demanding with no known theoretical guarantees for reliable estimation of $\beta^*$.  Stability selection \cite{meinshausen2010stability} is popular for high dimensional problems, but it is also computationally demanding.

In this paper, we develop an algorithm to select tuning parameters that is (i)~computationally efficient, (ii)~agnostic to the choice of the sparse regression method, and (iii)~asymptotically reliable.  Our proposed algorithm, called $\Path$, computes the solution path of a sparse regression method and thresholds a quantity computed at each point in the solution path.  We prove that, under appropriate conditions, $\Path$ is \textit{asymptotically tuning-free}, i.e., when the problem size becomes large (in the number of variables $p$ and the number of observations $n$), $\Path$ reliably estimates the location of the non-zero entries in $\beta^*$ independent of the choice of the threshold.  We compare $\Path$ to algorithms in the literature that use the Lasso to jointly estimate $\beta^*$ and the noise variance \cite{belloni2011square,stadler2011,sun2012scaled}.  We compliment our theoretical results with numerical simulations and also demonstrate the potential benefits of using $\Path$ in finite-dimensional settings.  

The rest of the paper is organized as follows.  Section~\ref{sec:ps} formulates the sparse regression problem.  Section~\ref{sec:path} presents the $\Path$ algorithm.  Section~\ref{sec:tuningfree} proves the asymptotic tuning-free property of $\Path$.  Section~\ref{sec:compare} compares $\Path$ to scaled Lasso \cite{sun2012scaled}, which is similar to square-root Lasso \cite{belloni2011square}.  Section~\ref{sec:numsim} presents numerical simulations on real data.  Section~\ref{sec:conclude} summarizes the paper and outlines some future research directions.

\section{Problem Formulation}
\label{sec:ps}

In this section, we formulate the sparse linear regression problem.  We assume that the observations $y \in \R^n$ and the measurement matrix $X \in \R^{n \times p}$ are known and related to each other by the linear model
\begin{equation}
y = X \beta^* + w \,, \label{eq:linmodel}
\end{equation}
where $\beta^* \in \R^p$ is the \textit{unknown sparse regression vector} that we seek to estimate.  We assume the following throughout this paper:
\begin{enumerate}[({A}1)]
\item The matrix $X$ is fixed with normalized columns, i.e., $\|X_i\|_2^2/n = 1$ for all $i \in \{1,2,\ldots,p\}$.
\item The entries of $w$ are i.i.d.\ zero-mean Gaussian random variables with variance $\sigma^2$.
\item  The vector $\beta^*$ is $k$-sparse with support set $S^* = \{j:\beta^*_j \ne 0\}$.
Thus, $|S^*| = k$.
%\item The number of observations $n$ is greater than or equal to $k$, i.e., $n \ge k$. % As we shall see later, this assumption is important to be able to compute the estimates of $\beta^*$.
\item The number of observations $n$ and the sparsity level $k$ are all allowed to grow to infinity as the number of variables $p$ grows to infinity.  In the literature, this is referred to as the \textit{high-dimensional framework}.
\end{enumerate}
For any set $S$, we associate a loss function, ${\cal L}(S;y,X)$, which is the cost associated with estimating $S^*$ by the set $S$.  An appropriate loss function for the linear problem in (\ref{eq:linmodel}) is the least-squares loss, which is defined as
\begin{align}
{\cal L}(S;y,X) \defn \min_{\alpha \in \R^{|S|}} \|y - X_{S} \alpha \|_2^2 = \bn \Pi^{\perp}[S] y \bn_2^2 \,, \label{eq:ls}
\end{align}
where $X_S$ is an $n \times |S|$ matrix that only includes the columns indexed by $S$ and $\Pi^{\perp}[S] = I - \Pi[S] = I - X_S ( X_S^T X_S)^{-1} X_S^T$ is the orthogonal projection onto the kernel of the matrix $X_S$.

In this paper, we mainly study the problem of estimating $S^*$, since, once $S^*$ has been estimated, an estimate of $\beta^*$ can be easily computed by solving a constrained least-squares problem.  Our main goal is to devise an algorithm for estimating $S^*$ that is \textit{asymptotically tuning-free} so that when $p \rightarrow \infty$, $S^*$ can be estimated with high probability without specifying a tuning parameter.

\section{Path Thresholding ($\Path$)}
\label{sec:path}

In this section, we develop $\Path$.  Recall that we seek to estimate the support of a sparse vector $\beta^*$ that is observed through $y$ and $X$ by the linear model in (\ref{eq:linmodel}).  Let $\Alg$ be a generic sparse regression algorithm with the following structure:
\begin{align}
\widehat{S}_s &= \Alg(y,X,s) \label{eq:alg}\\
y,X&= \text{Defined in (\ref{eq:linmodel})} \nonumber \\
s&= \text{Desired sparsity level} \nonumber\\
\widehat{S}_s &= \text{Estimate of $S^*$ s.t. $|\widehat{S}_s| = s$} \,. \nonumber
\end{align}
The tuning parameter in $\Alg$ is the sparsity level $s$, which is generally unknown.  All standard sparse regression algorithms can be defined using $\Alg$ with appropriate modifications.  For example, algorithms based on sparsity, including OMP \cite{tropp2007signal}, marginal regression \cite{fan2008sure,genovese2012comparison}, FoBa \cite{zhang2011adaptive}, and CoSaMP \cite{needell2009cosamp}, can be  written as (\ref{eq:alg}).  Algorithms based on real valued tuning parameters, such as the Lasso \cite{TibshiraniLasso1994}, can be written as (\ref{eq:alg}) after mapping the tuning parameter to the sparsity level.  We refer to Remark~\ref{rem:realRegul} for more details about this transformation.

\subsection{Motivating Example}

\begin{figure}
\begin{center}
\includegraphics[scale=0.52]{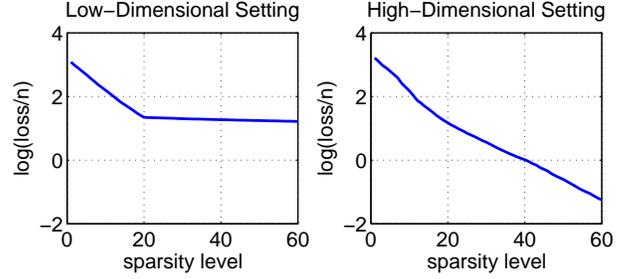}
\end{center}
\caption{Plot of a function of the loss versus the sparsity level when using the forward-backward (FoBa) algorithm in the low-dimensional setting (left) and the high-dimensional setting (right).}
\label{fig:path}
\end{figure}

Before presenting $\Path$, we discuss an example to illustrate the challenges in selecting tuning parameters for high-dimensional problems.  Let $p = 1000$, $k = 20$, and $X_{ij} \sim {\cal N}(0,\sigma^2)$, where $\sigma = 2$.  Furthermore, let the non-zero entries in $\beta^*$ be sampled uniformly between $[0.5,1.0]$.  Figure~\ref{fig:path} plots the $\log(\text{loss}/n)$ versus the sparsity level.  The loss at sparsity level $s$ is equal to $\|\Pi^{\perp}[\widehat{S}_s] y \|_2^2$ and $\widehat{S}_s$ is computed using the Forward-Backward (FoBa) sparse regression algorithm \cite{zhang2011adaptive}.   We refer to the sequence of estimates $\widehat{S}_1,\widehat{S}_2, \ldots, \widehat{S}_s$ as the \textit{solution path} of FoBa.  Furthermore, we consider a low-dimensional setting ($n = 2000$) and a high-dimensional setting ($n = 200$).  In both settings, $\widehat{S}_k = S^*$, i.e., FoBa outputs the true support.  

In the low-dimensional setting, we clearly see that Figure~\ref{fig:path} has a visible change at the sparsity level $s = 20$.  This suggests that the unknown sparsity  level could be inferred by appropriately thresholding some quantity computed over the solution path of a sparse regression algorithm.  However, in the high-dimensional setting, no such change is visible.  Thus, it is not clear if the unknown sparsity could be detected from the solution path.  As it turns out, we show in the next Section that an appropriate algorithm could be devised on the solution path to infer the sparsity level in an asymptotically reliable manner.

\subsection{Overview of $\Path$}
\SetNlSty{normal}{}{}
\IncMargin{1em}
\begin{algorithm}[h]
\caption{ Path Thresholding ($\Path$) }
\label{alg:path}
\smallskip
\hspace{-1em} {\textit{Inputs:} Observations $y$, measurement matrix $X$, and a parameter $c$} \\
\nl \For{$s = 0,1,2,\ldots,\min\{n,p\}$}{
\nl $\widehat{S}_s \gets \Alg(y,X,s)$ \\
\nl $\widehat{\sigma}^2_s \gets \| \Pi^{\perp}[\widehat{S}_{s}] y \|_2^2 /n$ \\
\nl $\displaystyle{\Delta_s \gets \max_{j \in (\widehat{S}_s)^c} \left\{  \| \Pi^{\perp}[\widehat{S}_s] y \|_2^2 - \|\Pi^{\perp}[\widehat{S}_s \cup j] y \|_2^2 \right\}}$ \\
\nl \If{ $\Delta_s < 2  c  \widehat{\sigma}^2_s \log p$}{
\nl Return $\widehat{S} = \widehat{S}_s$. \\
}}
\end{algorithm}
Algorithm~\ref{alg:path} presents path thresholding ($\Path$) that uses the generic sparse regression algorithm $\Alg$ in (\ref{eq:alg}) to estimate $S^*$.  Besides $y$ and $X$, the additional input to $\Path$ is a parameter $c$.  We will see in Section~\ref{sec:tuningfree} that as $p \rightarrow \infty$, under appropriate conditions, $\Path$ reliably identifies the true support as long as $c > 1$.

From Algorithm~\ref{alg:path}, it is clear that $\Path$ evaluates $\Alg$ for multiple different values of $s$, computes $\Delta_s$ defined in Line~4, and stops when $\Delta_s$ falls bellow a threshold.  The quantity $\Delta_s$ is the \textit{maximum} possible decrease in the loss when adding an additional variable to the support computed at sparsity level $s$.  The threshold in Line~5 of Algorithm~\ref{alg:path} is motivated from the following proposition.
\begin{proposition}
\label{prop:simple}
Consider the linear model in (\ref{eq:linmodel}) and assume that (A1)--(A4) holds.
If $\widehat{S}_k = S^*$, then $\Pr(\Delta_k < 2 c \sigma^2 \log p) \ge 1 - (p-k)/p^{c}$.
\end{proposition}
\begin{proof}
Using simple algebra, $\Delta_s$ in Line~4 of Algorithm~\ref{alg:path} can be written as
\begin{align}
\Delta_s = \max_{j \in (\widehat{S}_s)^c} \frac{|X_j^T \Pi^{\perp}[\widehat{S}_s] y |^2}{ \| \Pi^{\perp}[\widehat{S}_s] X_i \|_2^2} \,. \label{eq:deltas}
\end{align}
Under the conditions in the proposition, it is easy to see that 
$\Delta_{k} = \max_{j \in (S^*)^c} \| P_j w\|_2^2$, where $P_j$ is a rank one projection matrix.  The result follows from the Gaussian tail inequality and properties of projection matrices.
\end{proof}

Proposition~\ref{prop:simple} says that if $\sigma$ were known, then with high probability, $\Delta_k$ could be upper bounded by $2 c \sigma^2 \log p$, where $c$ is some constant.  Furthermore, under some additional conditions that ensure that $\Delta_s > 2c \sigma^2 \log p$ for $s < k$, Algorithm~\ref{alg:path} could estimate the unknown sparsity level with high probability.  For the marginal regression algorithm, the authors in \cite{genovese2012comparison} use a variant of this method to interchange the parameter $s$ and $\sigma^2$.

\begin{figure*}
\centering
\includegraphics[scale=0.6]{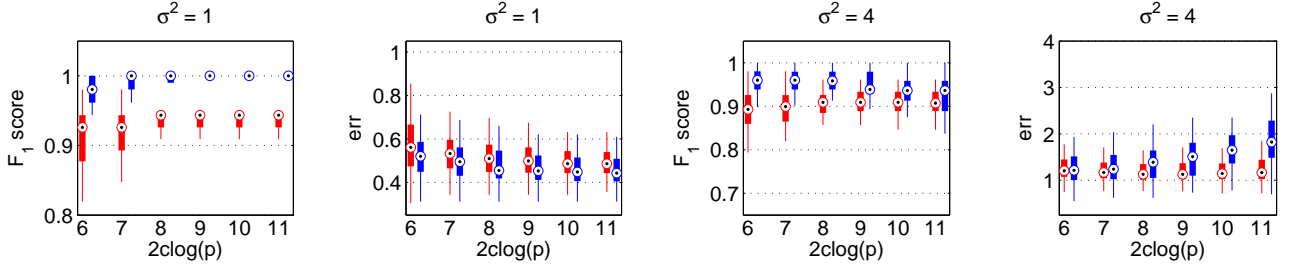}
\caption{An illustration of the performance of $\Path$ when $X \in \R^{216 \times 84}$ corresponds to stock returns data and $k = 30$.  The figures are box plots of error measures over $100$ trial.  Red is Lasso and blue is FoBa.  The first two plots show the $F_1$ score and the err when $\sigma^2 = 1$.  The last two plots show the $F_1$ score and err when $\sigma^2 = 4$.}
\label{fig:ill}
\end{figure*}

Since $\sigma$ is generally unknown, a natural alternative is to use an estimate of $\sigma$ to compute the threshold.  In $\Path$, we use an estimate of $\sigma$ computed from the loss at each solution of $\Alg$, i.e., $\widehat{\sigma}_s^2 = \| \Pi^{\perp}[\widehat{S}_s] y \|_2^2/n$.  Thus, for each $s$ starting at $s = 0$, $\Path$ checks if $\Delta_s \le 2  c  \widehat{\sigma}^2_s \log n$ and then stops the first time the inequality holds (see Line~5).  We can also use the estimate $\frac{\| \Pi^{\perp}[\widehat{S}_s] y \|_2^2}{n-s}$ with minimal change in performance since $s$ is generally much smaller than $n$.  Before illustrating $\Path$ using an example, we make some additional remarks regarding $\Path$.

\begin{remark}(Selecting the parameter $c$)
In Section~\ref{sec:tuningfree}, we prove that the performance of $\Path$ is independent of $c$  as $p\rightarrow \infty$ and $c > 1$.  However, in finite-dimensional settings, we need to set $c$ to an appropriate value.  Fortunately, the choice of $c$ is independent of the noise variance, which is \textit{not} the case for popular sparse regression algorithms like the Lasso, the OMP, and the FoBa.  In our numerical simulations, we observed that the performance of $\Path$ is insensitive to the choice of $c$ as long as $c \in [0.5,1.5]$.
\end{remark}

\begin{remark}(Computational Complexity)
Besides computing the solution path of $\Alg$, $\Path$ requires computing $\Delta_s$, which involves taking a maximum over at most $p$ variables.  For each $s$, assuming that that the residual $\Pi^{\perp}[\widehat{S}_s] y$ and the projection matrix $\Pi^{\perp}[\widehat{S}_s]$ is computed by the sparse regression algorithm, the additional complexity of $\Path$ is $O(n^2 p)$.  Furthermore, assuming that $\Path$ stops after computing $O(k)$ solutions, the total additional complexity of $\Path$ is $O(n^2 k p)$.  This complexity can be reduced to $O(n k p)$, with no change in the theoretical results that we present in Section~\ref{sec:tuningfree}, by modifying the $\Delta_s$ computations so that $\Delta_s = \max_{j \in (\widehat{S}_s)^c} | X_j^T \Pi^{\perp}[\widehat{S}_s] y |^2$.  This scaling of the additional complexity of $\Path$ is nearly the same as the complexity of computing the solution path of sparse regression algorithms.  For example, assuming that $p > n$, the solution path of the Lasso can be computed in time $O(n^2 p)$ using the LARS algorithm \cite{efron2004least}.
\end{remark}

\begin{remark}(Generalizing Algorithm~\ref{alg:path} to real valued tuning parameters)
\label{rem:realRegul}
We present $\Path$ in Algorithm~\ref{alg:path} using the sparsity level as a tuning parameter.  However, several sparse regression algorithms, such as the Lasso, depend on a real-valued tuning parameter.  One simple way to map such algorithms to the sparsity level is to compute the solution path, order the solutions in increasing level of sparsity, compute the loss at each solution, and finally select a solution at each sparsity level with minimal loss.  The last step ensures that there is a unique solution for each sparsity level.  We use this approach in all the implementations used in this paper.  Alternatively, depending on the sparse regression algorithm, we can also directly apply $\Path$ to the solution path.  For example, consider the Lasso that solves $\min_{\beta} [ \|y - X \beta \|_2^2/(2n) + s \| \beta \|_1 ]$.  As $s$ decreases, $|\widehat{S}_s|$ generally increases.  Thus, we can use Algorithm~\ref{alg:path} with the Lasso by simply replacing Line~1 with ``$\text{\textbf{for} }s = s_1,s_2,s_3,\ldots$", where $s_{i'} > s_{j'}$ for $i' < j'$.
\end{remark}

\subsection{Illustrative Example}
\label{subsec:illus}
In this section, we illustrate the performance of $\Path$.  We use data from \cite{choi2011learning} composed of $216$ observations of monthly stock returns from $84$ companies.  This results in a matrix $X$.  We let $k = 30$ and select a $\beta^*$ such that the non-zero entries in $\beta^*$ are uniformly distributed between $0.5$ and $1.5$.  We simulate two sets of observations; one using $\sigma^2 = 1$ and another one using $\sigma^2 = 4$.  We apply $\Path$ using Lasso \cite{TibshiraniLasso1994} and FoBa \cite{zhang2011adaptive}.  To evaluate the performance of an estimate $\widehat{\beta}$ with support $\widehat{S}$, we compute the $F_1$ score and the error in estimating $\beta^*$:
\begin{align}
F_1 \text{ score} &= 1/(1/\text{Recall} + 1/\text{Precision}) \,, \label{eq:f1}\\
\text{Precision} &= |S^* \cap \widehat{S}|/ |\widehat{S}| \,, \\
\text{Recall} &= |S^* \cap \widehat{S}|/ |S^*|  \,, \\
\text{err} &= \| \widehat{\beta} - \beta^* \|_2 \,. \label{eq:err}
\end{align}
Naturally, we want $F_1$ to be large (close to $1$) and $\text{err}$ to be small (close to $0$).  Figure~\ref{fig:ill} shows box plots of the error measures, where red corresponds to the Lasso and blue corresponds to FoBa.  The horizontal axis in the plots refer to the different choices of the threshold $2 c \log p$, where $c \in [0.6, 1.3]$.  For $\sigma$ small, we clearly see that the threshold has little impact on the final estimate.  For larger $\sigma$, we notice some differences, mainly for FoBa, as the threshold is varied.  Overall, this example illustrates that $\Path$ can narrow down the choices of the final estimate of $\beta^*$ or $S^*$ to a few estimates in a computationally efficient manner.  In the next section, we show that when $p$ is large, $\Path$ can accurately identify $\beta^*$ and $S^*$.

\section{Tuning-Free Property of $\Path$}
\label{sec:tuningfree}

In this section, we prove that, under appropriate conditions, $\Path$ is asymptotically tuning-free.  We state our result in terms of the generic sparse regression algorithm $\Alg$ defined in (\ref{eq:alg}).  For a constant $c > 1$, we assume that $\Alg$ has the following property:
\begin{enumerate}[(A5)]
\item $\Alg$ can reliably estimate the true support $S^*$, i.e., $\Pr(S^* = \Alg(y,X,k)) \ge 1 - 1/p^{c}$, where $c$ is the input to $\Path$.
\end{enumerate}
Assumption (A5), which allows us to separate the analysis of $\Path$ from the 
analysis of $\Alg$, says that $\Alg$ outputs the true support for $s = k$ with high probability.  
Under appropriate conditions, this property holds for all sparse regression algorithms under various conditions \cite{buhlmann2011statistics}.
The next theorem is stated in terms of the following two parameters:
\begin{align}
\beta_{\min} &= \min_{i \in S^*} | \beta_i | \,, \\
\rho_{2k} &= \min\left\{ \frac{\|X_A v \|_2^2}{n \|v\|_2^2} : S^* \subseteq A, v \in \R^{2k}  \right\} \,. \label{eq:rerhp}
\end{align}
The parameter $\beta_{\min}$ is the minimum absolute value over the non-zero entries in $\beta^*$.  The parameter $\rho_{2k}$, referred to as the restricted eigenvalue (RE), is the minimum eigenvalue over certain blocks of the matrix~$X^T X/n$.
\begin{theorem}
\label{thm:path}
Consider the linear model in (\ref{eq:linmodel}) and assume that (A1)--(A4) holds.  Let $\widehat{S}$ be the output of $\Path$ used with the sparse regression method $\Alg$ defined in (\ref{eq:alg}) that satisfies (A5).  Select $\epsilon > k/n + \sqrt{1/k}$ and $c > 1/(1-\epsilon)$.  For some constant $C > 0$, if
\begin{align}
n \ge \frac{2 c k \log p }{\rho_{2k}} + \frac{8 \sigma c \sqrt{k} \log p }{\beta_{\min} \rho_{2k}^2} + \frac{8 \sigma^2 c \log p}{\beta_{\min}^2 \rho_{2k}^2}\,, \label{eq:nn}
\end{align}
then $\Pr(\widehat{S} = S^*) \ge 1 - C p^{ 1 - (1-\epsilon) c } $.
\end{theorem}
The proof of Theorem~\ref{thm:path}, given in Appendix~\ref{sec:tuningfreeproof}, simply identifies sufficient conditions under which $\Delta_s \ge 2 c \log p$ for $s < k$ and $\Delta_k < 2 c \log p$.  We now make some remarks regarding Theorem~\ref{thm:path}.
\begin{remark}(Tuning-Free Property)
Recall from (A4) that $k,n \rightarrow \infty$ as $p \rightarrow \infty$.  Furthermore, based on the conditions on $n$ in (\ref{eq:nn}), it is clear that $\lim_{p \rightarrow \infty} (k/n + \sqrt{1/k}) = 0$.  This means that for any $c > 1$, if (\ref{eq:nn}) holds, then $\lim_{p \rightarrow \infty} \Pr(\widehat{S} = S^*) \rightarrow 1$.  This shows that $\Path$ is \textit{asymptotically tuning-free}.
\end{remark}

\begin{remark}({Numerical Simulation})
To illustrate the tuning-free property of $\Path$, we consider a simple numerical example.  Let $p = 1000$, $k = 50$, and $\sigma = 1$.  Assume that the non-zero entries in $\beta^*$ are drawn from a uniform distribution on $[1,2]$.  Next, randomly assign each non-zero entry either a positive or a negative sign.  The rows of the measurement matrix $X$ are drawn i.i.d.\ from ${\cal N}(0,\Sigma)$.  We consider two cases: $\Sigma = I$ and $\Sigma = 0.8I + 0.2\textbf{1}  \textbf{1}^T$, where $\textbf{1}$ is a vector of ones.  For both cases, we use $\Path$ with Lasso and FoBa.  Figure~\ref{fig:tuningfree} plots the mean $F_1$ score, defined in (\ref{eq:f1}),  and mean $\log(\text{err})$, defined in (\ref{eq:err}), over $100$ trials as $n$ ranges from $200$ to $1000$.  Note that $F_1 = 1$ corresponds to accurate support recovery.  For both Lasso and FoBa, we used $\Path$ with $c = 1$ and $c = 1.5$.  We clearly see that once $n$ is large enough, both algorithms, with different choices of $c$, lead to accurate support recovery.  This naturally implies accurate estimation of $\beta^*$.
\end{remark}

\begin{figure}
\centering
\includegraphics[scale=0.525]{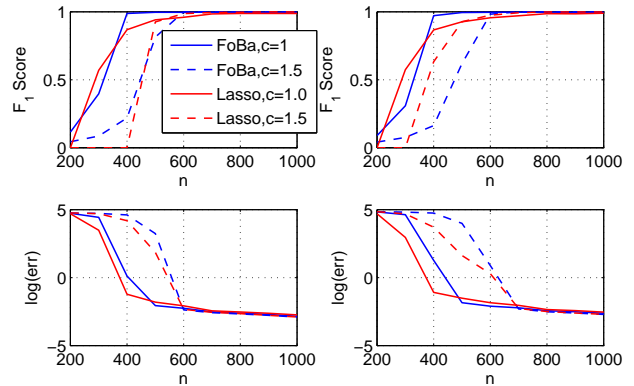}
\caption{ Mean $F_1$ score and mean $\log(\text{err})$ over $100$ trials when using $\Path$ with Lasso and FoBa with $p = 1000$, $k = 10$, $\sigma = 1$, and $\beta_{\min} \ge 1$. (left) $X$ is drawn from ${\cal N}(0,I)$ (right) $X$ is drawn from  $\sim {\cal N}(0, 0.8I + 0.2\textbf{1}  \textbf{1}^T)$.
}
\label{fig:tuningfree}
\end{figure}

%\begin{remark}
%\label{rem:sc}
%We can replace $2c \log p$ in (\ref{eq:nn}) by any function of $n$ and $p$, say $\tau_{n,p}$, such that $\P( z^2 < \tau_{n,p}) \rightarrow 1$, where $z \sim {\cal N}(0,1)$.  The only change will be appropriately modifying the threshold in Algorithm~\ref{alg:path} and error rate in Theorem~\ref{thm:path}.
%Further, the error rate in (\ref{eq:nn}) can be further improved to $1 - C/n^{c-\epsilon}$, where $\epsilon > 0$ is arbitrary small.  In finite-dimensional settings, a small $c$ results in overestimate and a large $c$ results in underestimation.  We recommend selecting $c \in [0.5,1.5]$.
%\end{remark}

\begin{remark}(Superset Recovery)
We have analyzed $\Path$ for reliable support recovery.  If $\Alg$ can only output a superset of $S^*$, which happens when using the Lasso under the restricted strong convexity condition \cite{bickel2009simultaneous,negahban2010unified}, then (A5) can be modified and the statement of Theorem~2 can be modified to reflect superset recovery.  The only change in (\ref{eq:nn}) will be to appropriately modify the definition of $\rho_{2k}$ in (\ref{eq:rerhp}).
\end{remark}

\begin{remark}(Extension to sub-Gaussian noise)
Our analysis only uses tail bounds for Gaussian and chi-squared random variables.  Hence, we can easily extend our analysis to sub-Gaussian noise using tail bounds for sub-Gaussian and sub-exponential random variables in \cite{VershyninRMT}.
\end{remark}

\begin{remark}(Scaling of $n$)
For an appropriate parameter $c_2$ that depends on $\rho_{2k}$, $\sigma$, and $\beta_{\min}$, (\ref{eq:nn}) holds if $n > c_2 k \log p$.  When using the Lasso, under appropriate conditions, in the most general case, $n > c_2' k \log p$ is sufficient for reliable support recovery.  Thus, both the Lasso and the $\Path$ require the same scaling on the number of observations.  An open question is to analyze the tightness of the condition (\ref{eq:nn}) and see if the dependence on $k$ in (\ref{eq:nn}) can be improved.  One way may be to incorporate prior knowledge that $k \ge k_{\min}$.  In this case, the scaling in (\ref{eq:nn}) can be improved to $n > c_2 (k-k_{\min}) \log p$.
\end{remark}

\section{Connections to Scaled Lasso}
\label{sec:compare}

\begin{figure*}
\includegraphics[scale=0.6]{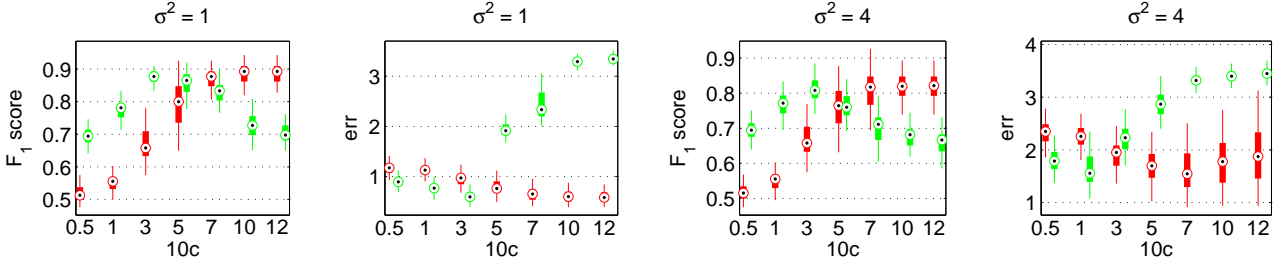}
\caption{A comparison of Lasso (blue) to scaled lasso (green).  See Figure~\ref{fig:ill} for details about the setup. }
\label{fig:comp}
\end{figure*}

In this Section, we highlight the connections between $\Path$ and the scaled Lasso algorithm proposed in \cite{sun2012scaled}.  Let $\widehat{\sigma}$ be an initial estimate of the noise variance.  Consider the following computations on the solution path until an equilibrium is reached:
\begin{align}
&\text{$t \gets \min_{s} \{ s : \Delta_s < 2c \widehat{\sigma}^2 \log p \}$} \label{eq:oo1}\\
&\text{$\widehat{S}_{t} \gets \Alg(y,X,t)$}\\
&\text{$\widehat{\sigma}_{new} \gets \| \Pi^{\perp}[\widehat{S}_{t}] y \| / \sqrt{n}$}\\
&\text{If $\widehat{\sigma}_{new} \ne \widehat{\sigma}$, then let $\widehat{\sigma} \gets \widehat{\sigma}_{new}$ and go to (\ref{eq:oo1}).} \label{eq:oo4}
\end{align}
The next theorem shows that, under an additional condition on the solution path, (\ref{eq:oo1})-(\ref{eq:oo4}) is equivalent to implementing $\Path$.
\begin{theorem}
\label{thm:scal}
Suppose $\widehat{\sigma}_0 > \widehat{\sigma}_1 > \cdots > \widehat{\sigma}_{n-1} > \widehat{\sigma}_n$, where $\widehat{\sigma}_s = \| \Pi^{\perp}[\widehat{S}_s] y \|_2^2/n$.  If $\widehat{\sigma} = \widehat{\sigma}_0 = \|y\|_2/\sqrt{n}$, then the output of $\Path$ is equal to the output of (\ref{eq:oo1})-(\ref{eq:oo4}).
\end{theorem}
\begin{proof}
Let $\widehat{S}$ be the output of $\Path$ and $\widehat{T}$ be the output of (\ref{eq:oo1})-(\ref{eq:oo4}).  From the properties of $\Path$, we know that $\Delta_s > 2c\widehat{\sigma}_s^2 \log n$ for all $s < |\widehat{S}|$. Furthermore, we know that $\widehat{\sigma}_1 > \widehat{\sigma}_2 > \cdots > \widehat{\sigma}_s$.  This means that $\Delta_s > 2 c \widehat{\sigma}_t^2 \log n$ for any $t \le s < |\widehat{S}|$.  Thus, (\ref{eq:oo1})--(\ref{eq:oo4}) reaches an equilibrium when $\widehat{T} = \widehat{S}$.
\end{proof}

The condition in Theorem~\ref{thm:scal} ensures that the loss decreases as the tuning parameter $s$ in $\Alg$ increases.  This condition easily holds for sparse regression algorithms, including the OMP and the FoBa, that are based on greedy methods.

Interestingly, (\ref{eq:oo1})--(\ref{eq:oo4}) resemble the scaled Lasso algorithm \cite{sun2012scaled}.  In particular, scaled Lasso starts with an initial estimate of $\sigma$, selects an appropriate tuning parameter, selects an estimate from the solution path of Lasso, and repeats until an equilibrium is reached.  Furthermore, as shown in \cite{sun2012scaled}, scaled Lasso solves the following problem:
\begin{equation}
(\widehat{\beta}, \widehat{\sigma}) = \argmin_{\beta \in \R^p,\sigma > 0} \left[\frac{\| y - X \beta\|_2^2}{2n \sigma} + \frac{\sigma}{2} + \lambda \|\beta\|_1\right] \,. \label{eq:hbl1}
\end{equation}
The theoretical results in \cite{sun2012scaled} show that, under appropriate restricted eigenvalue conditions, selecting $\lambda_0 = \sqrt{2 c \log p/n}$, where $c$ is a constant, results in accurate estimation of $\beta^*$.  This choice of $\lambda_0$ comes from the upper bound of the random variable $\|X^T w\|_{\infty}/(\sigma n)$.  The motivation for scaled Lasso grew out of the results in \cite{stadler2011} and related discussions in \cite{SunComments2010,Antoniadis2010}.  Furthermore, as shown in \cite{giraud2012high}, scaled Lasso is equivalent to square-root Lasso \cite{belloni2011square}.

$\Path$ can be seen as a solution to the more general problem
\begin{equation}
(\widehat{\beta},\widehat{\sigma}) = \argmin_{ \beta \in\R^p, \|\beta\|_0 \le s, \sigma > 0} \left[\frac{\| y - X \beta\|_2^2}{2n \sigma} + \frac{\sigma}{2} \right] \,, \label{eq:hbl0}
\end{equation}
where $\|\beta\|_0$ is the $\ell_0$-norm that counts the number of non-zero entries in $\beta$ and the parameter $s$ is implicitly related to the threshold $2c \log p$ in Line~4 of $\Path$ (Algorithm~\ref{alg:path}).  The advantage of $\Path$ is that it can be used with \textit{any} sparse regression algorithm and the choice of the threshold does not depend on the measurement matrix.

%In particular, while the choice of $\lambda_0$ in the scaled Lasso depends critically on the number of variables $p$ and the usage of the $\ell_1$ norm as an approximation to the $\ell_0$ norm, the choice of $s$ in (\ref{eq:hbl0}) only depends on the statistics of a normal random variable (see Remark~\ref{rem:sc}).  This readily allows us to use $\Path$ and $\Path$ with \textit{any} sparse regression algorithm. 

We now empirically compare scaled Lasso (SL) to $\Path$ used with Lasso (PL).
In particular, we want to compare the sensitivity of both methods to the choice of the parameter $\lambda_0 = \sqrt{2c \log p/n}$ in SL and the choice of the parameter $2c\log n$ in PL.  We consider the same setting as in Figure~\ref{fig:ill} (see Section~\ref{subsec:illus} for the details).  Figure~\ref{fig:comp} shows the box plots for the $F_1$ score and err for PL (in red) and SL (in green) as $c$ ranges from $0.05$ to $1.2$.  It is clear that there exists a parameter $c$ for both SL and PL such that their performance is nearly equal.  For example, in the first plot, the $F_1$ score of SL at $c = 0.5$ is nearly the same as the $F_1$ score of PL at $c = 1.0,1.2$.  Furthermore, it is clear from the plot that PL is relatively insensitive to the choice of $c$ when compared to SL.  This suggests the possible advantages of using $\Path$ with Lasso over the scaled Lasso estimator.

\section{Application to Real Data}
\label{sec:numsim}
\vspace{-0.3cm}

\begin{figure*}
\includegraphics[scale=0.6]{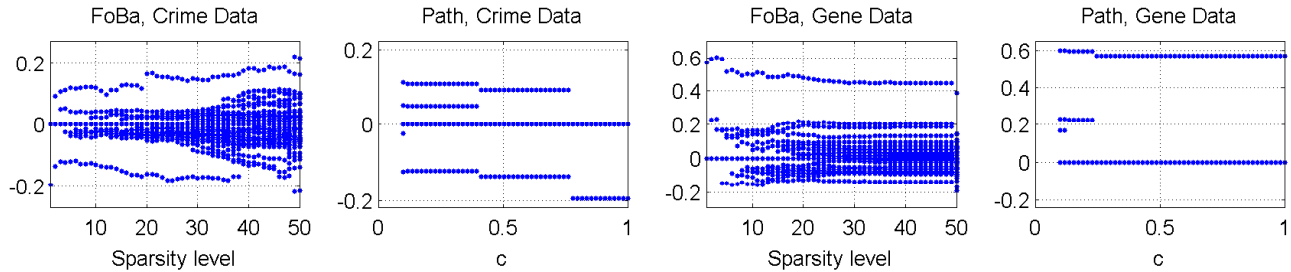}
\caption{Comparison of the solution paths of FoBa and $\Path+\!$~FoBa.  The first figure shows the solution path of FoBa applied to the crime data.  The horizontal axis specifies the sparsity level and the vertical axis specifies the coefficient values.  The second figure applies 
$\Path$ to the solution path in the first figure for $50$ different values of $c$ in the range $[0.1,1]$. $\Path$ reduces the total number of solutions from $50$ to $4$.  We observe similar trends for the gene data (last two figures).}
\vspace{-0.2cm}
\label{fig:fobareal}
\end{figure*}

In this Section, we demonstrate the advantages of using $\Path$ on real data.  Although, we have shown that $\Path$ is asymptotically tuning-free, this may not be true in finite-dimensional settings.  Instead, our main goal is to demonstrate how $\Path$ can significantly reduce the possible number of solutions for various regression problems by reparametrizing the sparse regression problem in terms of the parameter $c$ in $\Path$.
We consider the following two real data sets:
\begin{enumerate}
\item[$\bullet$] UCI Communities and Crime Data \cite{redmond2002data,BacheLichman2013}:  This data set contains information about the rate of violent crimes from $1994$ communities.  The goal is to find a sparse set of attributes, from the $122$ given attributes, that best predict the rate of violent crimes.  We removed $22$ attributes due to insufficient data, randomly selected $100$ communities, and normalized the columns to get a matrix $X \in \R^{100 \times 100}$ and observations $y \in \R^{100}$.
\item[$\bullet$] Prostate cancer data \cite{ProsData}: This data set contains information about gene expression values of $12533$ genes from $50$ patients with prostate cancer.  We consider the problem of finding relationships between the genes.  For this, we randomly select the expression values from one gene, say $y \in \R^{50}$, and regress on the remaining gene expression values $X \in \R^{50 \times 12532}$.
\end{enumerate}
Figure~\ref{fig:fobareal} presents our results on applying FoBa \cite{zhang2011adaptive} to both the data sets. % Results for Lasso are shown in the supplementary material.  
As detailed in the caption of Figure~\ref{fig:fobareal}, $\Path$ reduces the number of solutions from $50$ to $4$ for each data set without tuning any parameter.  Subsequently, cross-validation or stability selection \cite{meinshausen2010stability} can be used to further prune the set of estimates.  This post-processing stage will require far less computations than applying cross-validation or stability selection without using $\Path$ to reduce the number of solutions.

\section{Conclusions}
\label{sec:conclude}
\vspace{-0.3cm}
We have proposed a new computationally efficient algorithm, called path thresholding ($\Path$), for selecting tuning parameters in any high-dimensional sparse regression method.  Our main contribution shows that (i) $\Path$ is agnostic to the choice of the sparse regression method, and (ii) $\Path$, under appropriate conditions, selects the optimal tuning parameter with high probability as the problem size grows large.  Thus, using $\Path$ with any tuning-dependent regression algorithm leads to an asymptotically tuning-free sparse regression algorithm.  In finite-dimensional settings, we have shown that $\Path$ can drastically reduce the possible number of solutions of a sparse regression problem.  Thus, $\Path$ can be used to significantly reduce the computational costs associated with cross-validation and stability selection.

Our work motivates several avenues of future research.  For example, it will be interesting to explore the use of $\Path$ for estimating approximately sparse vectors.  Furthermore, it will also be useful to study the use of $\Path$ in asymptotically tuning-free learning of graphical models.

\section*{Acknowledgment}
\vspace{-0.3cm}
We thank Prof. Dennis Cox and Ali Mousavi for feedback that improved the quality of the paper.  This work was supported by the Grants NSF IIS-1124535, CCF-0926127, CCF-1117939; ONR N00014-11-1-0714, N00014-10-1-0989; and ARO MURI W911NF-09-1-0383.

\appendix

\section{Proof of Theorem~\ref{thm:path}}
\label{sec:tuningfreeproof}
\vspace{-0.3cm}
We want to show that $\Path$ stops when $s = k$, where recall that $k$ is the unknown sparsity level of the sparse vector $\beta^*$ that we seek to estimate.  For this to happen, it is sufficient to find conditions under which the following two equations are true:
\begin{align}
\Delta_s &\ge \widehat{\sigma}_s \tau_p \,, \text{whenver } s < k \,, \label{eq:nostop}\\
\Delta_k &< \widehat{\sigma}_k \tau_p \,, \label{eq:stop}
\end{align}
where $\tau_p = 2c \log p$, $\widehat{\sigma}_s^2 = \| \Pi^{\perp}[\widehat{S}_s] y \|_2^2/n$, and $\Delta_s$ is defined in Line~4 of Algorithm~1.  It is clear that if (\ref{eq:nostop}) holds, then Algorithm~\ref{alg:path} will not stop for $s < k$.  Furthermore, if (\ref{eq:stop}) holds, then Algorithm~\ref{alg:path} stops when $s = k$.  The rest of the proof is centered around finding conditions under which both (\ref{eq:nostop}) and (\ref{eq:stop}) hold. 

\textbf{Finding conditions under which (\ref{eq:nostop}) holds.}  Using (\ref{eq:deltas}), we can lower bound $\Delta_s$ as follows:
\begin{align}
\Delta_s &= \!\! \max_{j \in (\widehat{S}_s)^c} \! \frac{|X_j^T \Pi^{\perp}[\widehat{S}_s] y|^2}{\|\Pi^{\perp}[\widehat{S}_s] X_j \|_2^2} 
\overset{(a)}{\ge} \frac{1}{n}\max_{j \in S^* \backslash \widehat{S}_s} {|X_j^T \Pi^{\perp}[\widehat{S}_s] y|^2} \nonumber \\
&\overset{(b)}{=} \frac{1}{n} \left\|X_{S^* \backslash \widehat{S}_s}^T \Pi^{\perp}[\widehat{S}_s] y\right\|_{\infty}^2 \nonumber \\
&\overset{(c)}{\ge} \frac{1}{k_s n} \left\|X_{S^* \backslash \widehat{S}_s}^T \Pi^{\perp}[\widehat{S}_s] y\right\|_{2}^2 \,, \quad k_s = |S^* \backslash \widehat{S}_s| \nonumber \\
&\overset{(d)}{\ge} \frac{1}{k_s n} \left[ \left\|{\cal R}\right\|_{2}^2 -
2 \left|({\cal R})^T w \right| \right] \,. \label{eq:leftf1}
\end{align}
Step~(a) restricts the size over which the maximum is taken and uses the equation $\|\Pi^{\perp}[\widehat{S}_s] X_j \|_2^2 \le n$.  Step~(b) uses the $\ell_{\infty}$-norm notation.  Step~(c) uses the fact that $\|v\|_{\infty}^2 \ge \|v\|_2^2/k_s$ for any $k_s \times 1$ vector $v$.
Step~(d) uses (\ref{eq:linmodel}) and also defines the notation
\begin{equation}
{\cal R} = X_{S^* \backslash \widehat{S}_s}^T \Pi^{\perp}[\widehat{S}_s] X_{S^* \backslash \widehat{S}_s} \beta^*_{S^* \backslash \widehat{S}_s} \,. \label{eq:calr}
\end{equation}
Next, we use (\ref{eq:linmodel}) to evaluate $\widehat{\sigma}_s^2$:
\begin{align}
n \widehat{\sigma}_s^2 = \| \Pi^{\perp}[\widehat{S}_s] y \|_2^2
&\le \| \Pi^{\perp}[\widehat{S}_s] X \beta^* \|_2^2 
+ \| \Pi^{\perp}[\widehat{S}_s] w \|_2^2 \nonumber \\
&+ 2 | (\Pi^{\perp}[\widehat{S}_s] X \beta^*)^T w | \,. \label{eq:noiseb1}
\end{align}
Using (\ref{eq:leftf1}) and (\ref{eq:noiseb1}), (\ref{eq:nostop}) holds if the following equation holds for $s < k$:
\begin{align}
&\frac{\left\|{\cal R}\right\|_{2}^2}{\tau_p k_s } 
- {\| \Pi^{\perp}[\widehat{S}_s] X \beta^* \|_2^2}  \label{eq:nostop1}\\
&\ge \frac{2 \left|({\cal R})^T w \right|} {\tau_p k_s }  
+ \| \Pi^{\perp}[\widehat{S}_s] w \|_2^2 
 + 2 | (\Pi^{\perp}[\widehat{S}_s] X \beta^*)^T w | \,. \nonumber
\end{align}
We now find a lower bound for the left hand side (LHS) of (\ref{eq:nostop1}) and an upper bound for the right hand side (RHS) of (\ref{eq:nostop1}).  Using (\ref{eq:calr}), we have
\[
\|{\cal R}\|_2^2 = \left(\beta^*_{S^* \backslash \widehat{S}_s}\right)^T \left(X_{S^* \backslash \widehat{S}_s}^T \Pi^{\perp}[\widehat{S}_s] X_{S^* \backslash \widehat{S}_s}\right)^2 \beta^*_{S^* \backslash \widehat{S}_s} \,.
\]
Let $0 < \lambda_1 \le \lambda_2 \le \ldots \le \lambda_{k_s}$ be the eigenvalues of $X_{S^* \backslash \widehat{S}_s}^T \Pi^{\perp}[\widehat{S}_s] X_{S^* \backslash \widehat{S}_s}$.  Then, for a unitary matrix $M$ and a diagonal matrix $D$, where $D_{ii} = \lambda_i$, $X_{S^* \backslash \widehat{S}_s}^T \Pi^{\perp}[\widehat{S}_s] X_{S^* \backslash \widehat{S}_s} = M^T D M$.  If $\xi = M \beta^*_{S^* \backslash \widehat{S}_s}$, then the LHS of (\ref{eq:nostop1}) can be written as
\begin{align}
\frac{ \xi^T D^2 \xi }{\tau_p k_s } - \xi^T D \xi 
&= \sum_{i=1}^{k_s} \left[ 
\frac{\lambda_i^2 }{\tau_p k_s } - \lambda_i
\right] \xi_i^2 \nonumber \\
&\ge \left[ \frac{n^2 \rho_{2k}^2 }{\tau_p k_s } - n \rho_{2k} \right] \|\beta^*_{S^* \backslash \widehat{S}_s}\|_2^2 \,,
\end{align}
where we use the assumption that $n \rho_{2k} > \tau_p k_s $ and the fact that $\lambda_1 > n \rho_{2k}$, where $\rho_{2k}$ is defined in (\ref{eq:rerhp}).

Next, we find an upper bound for the RHS in (\ref{eq:nostop1}).  To do so, we make use of standard results for Gaussian random vectors.  In particular, we have that for any vector $v$, 
\begin{align}
&\Pr(|v^T w| > \sigma \|v\|_2 \sqrt{\tau_p}) \le 2 e^{-\tau_p/2} \,, \label{eq:dd1}\\
&\Pr( \| \Pi^{\perp}[\widehat{S}_s] w\|_2^2/\sigma^2 \ge n-s + \sqrt{(n-s) \tau_p/2} +\tau_p)  \nonumber\\ 
&\hspace{5cm}\le e^{-\tau_p/2} . \label{eq:dd2}
\end{align}
Using (\ref{eq:dd1}) and (\ref{eq:dd2}), we can upper bound the RHS in (\ref{eq:nostop1}) using the following equation with probability at least $1 - 3 e^{-\tau_p/2}$:
\begin{align}
&\frac{2\sigma \|{\cal R}\|_2 \sqrt{\tau_p}}{\tau_p k_s } 
+ 2 \sigma 1\| \Pi^{\perp}[\widehat{S}_s] X \beta^*\|_2 \sqrt{\tau_p} \nonumber \\
&\qquad + \sigma^2\left( n - s + \sqrt{(n-s) \tau_p/2} +\tau_p \right) \label{eq:rft}
\end{align}
Next, note that $\|{\cal R}\|_2 \le k_s n \| \beta^*_{S^*\backslash \widehat{S}_s}\|_2$ and  
$\| \Pi^{\perp}[\widehat{S}_s] X \beta^*\|_2 \le \sqrt{nk_s} \| \beta^*_{S^*\backslash \widehat{S}_s}\|_2$.  Furthermore, choosing $\tau_p/2 < n$, (\ref{eq:rft}) can be upper bounded by
\[
\frac{2\sigma n \| \beta^*_{S^*\backslash \widehat{S}_s}\|_2 \sqrt{\tau_p} }{\tau_p}
+ 2 \sigma \sqrt{k_s n} \| \beta^*_{S^*\backslash \widehat{S}_s}\|_2 \sqrt{\tau_p}
+ 4 \sigma^2 n \,.
\]
Rearranging terms, using $\|\beta^*_{S^*\backslash \widehat{S}_s}\|_2^2 \ge k_s \beta_{\min}^2$, $n \ge \tau_p k_s$, and $k_s \le k$, (\ref{eq:nostop1}) holds with probability at least $1 - 3 e^{-\tau_p/2}$ if the following holds:
\begin{align}
n \ge \frac{k \tau_p }{\rho_{2k}}  + \frac{4 \sigma^2 \tau_p}{\beta_{\min}^2 \rho_{2k}^2} 
+ \frac{4 \sigma \sqrt{k} {\tau_p}}{\beta_{\min} \rho_{2k}^2}
\,. \label{eq:nostop2}
\end{align}

\textbf{Finding conditions under which (\ref{eq:stop}) holds.}  Assuming that $\widehat{S}_k = S^*$ and using the definition of $\Delta_k$ and $\widehat{\sigma}_k$, we can write (\ref{eq:stop}) as
\begin{align}
\max_{j \in (S^*)^c} \| P_i w \|_2^2/ \sigma^2 < \|\Pi^{\perp}[S^*] w\|_2^2 \tau_p / (n \sigma^2) \,,
\end{align}
where $P_i$ is a rank one projection matrix.  Note that $\|P_i w\|_2^2/\sigma^2$ is the square of a ${\cal N}(0,1)$ random variable.  Using the Gaussian tail ineqaulity, we have that
\[
\Pr\left( \max_{j \in (S^*)^c} \| P_i w \|_2^2/ \sigma^2 \ge \nu_p\right) \le 2(p-k) e^{-\nu_p/2}/\sqrt{\nu_p} \,.
\]
Moreover, using standard bounds for chi-square random variables, we have that
\[
\Pr\left( \|\Pi^{\perp}[S^*] w\|_2^2/\sigma^2 \le n-k - 2\sqrt{(n-k)\alpha_p}\right) \le e^{-\alpha_p} \,.
\]
Thus, (\ref{eq:stop}) holds with probability at least $1 - 3(p-k) e^{-\nu_p/2} / \sqrt{\nu_p}$ if the following is true
\[
\frac{\nu_p}{\tau_p} < 1 - \frac{k}{n} - 2\frac{\sqrt{(n-k)\nu_p/2}}{n}
\]
Let $\nu_p = 2 (1-\epsilon) c \log p$.  Then, the above conditions holds if $
\epsilon > \frac{k}{n} + \sqrt{\frac{2c \log p}{n}}$.  Since $n > 2ck \log p$, if $\epsilon > k/n + \sqrt{1/k}$, then (\ref{eq:stop}) holds with probability at least $1-3p^{1-(1-\epsilon)c}/\sqrt{2(1-\epsilon)c\log p}$.

Combining both results, and using (A5), we have that under the conditions stated in the theorem, for some constant $C > 0$, $\Pr(\widehat{S} = S^*) \ge 1 -  {C p^{1-(1-\epsilon)c}}$.

\end{document}